\documentclass{amsart}
\usepackage{fullpage,hyperref,tikz-cd}
\usepackage[all]{xy}
\usepackage[normalem]{ulem}

\newtheorem{theorem}{Theorem}[section]
\newtheorem{lemma}[theorem]{Lemma}

\newtheorem{algorithm}[theorem]{Algorithm}
\theoremstyle{definition}
\newtheorem{defn}[theorem]{Definition}
\newtheorem{remark}[theorem]{Remark}

\newtheorem{example}[theorem]{Example}
\newtheorem{proposition}[theorem]{Proposition}

\newcommand{\FF}{\mathbb{F}}
\newcommand{\QQ}{\mathbb{Q}}

\newcommand{\ZZ}{\mathbb{Z}}
\newcommand{\HH}{\mathbb{H}}
\newcommand{\CC}{\mathbb{C}}

\newcommand{\PP}{\mathbf{P}}

\newcommand{\calM}{\mathcal{M}}

\newcommand{\calS}{\mathcal{S}}

\newcommand{\slz}{\SL_2(\ZZ)}

\DeclareMathOperator{\Div}{Div}

\DeclareMathOperator{\GL}{GL}
\DeclareMathOperator{\SL}{SL}
\DeclareMathOperator{\Gal}{Gal}

\DeclareMathOperator{\ns}{ns}

\newcommand{\Magma}{\textsc{Magma}{}}
\newcommand{\SageMath}{\textsc{SageMath}}

\newcommand{\arXiv}[3]{\href{https://arxiv.org/abs/#1}{arXiv:#1v#2} (#3)}

\newenvironment{psmallmatrix}
  {\left(\begin{smallmatrix}}
  {\end{smallmatrix}\right)}

\begin{document}
\title{Coleman Integration on Modular Curves}
\author{Mingjie Chen, Kiran S. Kedlaya, and Jun Bo Lau}
\thanks{Kedlaya was supported by NSF (grants DMS-2053473 and prior), Simons Foundation (Simons Fellowship in Mathematics 2023--2024), and UC San Diego (Warschawski professorship).}

\date{January 2024}

\begin{abstract}
Coleman integrals is a major tool in the explicit arithmetic of algebraic varieties, notably in the study of rational points on curves. One of the inputs to compute Coleman integrals is the availability of an affine model. We develop a model-free algorithm that computes single Coleman integrals between any two points on modular curves. Using Hecke operators, any Coleman integral can be broken down into a sum of tiny integrals. We illustrate this using several examples computed in \SageMath{} and \Magma. We also suggest some future directions for this work, including a possible extension to iterated Coleman integrals.
\end{abstract}

\maketitle


\section{Introduction}\label{section1}

In the 1980s, Coleman developed a theory of $p$-adic line integrals on curves
and higher dimensional varieties with good reduction at $p$
\cite{coleman-de-shalit, coleman-ann}. This has found many applications, e.g., computing torsion points on curves \cite{coleman-ann}, $p$-adic regulators \cite{coleman-de-shalit} and heights \cite{coleman-gross}, and rational points on curves via the \emph{Chabauty--Coleman method} \cite{coleman-chabauty}. 

In the 2010s, it became feasible to make numerical computations of Coleman integrals on individual curves via Dwork’s principle
of analytic continuation along Frobenius. Building on Kedlaya's algorithm for computing the action of Frobenius on the Monsky--Washnitzer cohomology of odd degree hyperelliptc curves \cite{kedlaya-mw},
a Coleman integration algorithm for such curves was introduced by Balakrishnan--Bradshaw--Kedlaya \cite{bbk10}. This was subsequently generalized to arbitrary curves by Balakrishnan--Tuitman \cite{bt-coleman}; it should be emphasized that this algorithm requires an input an explicit projective representation of the target curve (namely, a singular plane model).

In this paper, we specialize the general theory of Coleman integration to the case of \emph{modular curves}. A modular curve is associated to a \emph{congruence subgroup} $H$ of $\GL_2(\ZZ)$, namely one containing the kernel of $\GL_2(\ZZ) \to \GL_2(\ZZ/N\ZZ)$ for some positive integer $N$; for any number field $K$,
the $K$-rational points of the associated modular curve $X_H$ (roughly) classify elliptic curves over $K$ such that the Galois action on the $N$-torsion points of $E(\overline{K})$ factors through $H$.
Consequently, classifying $K$-rational points on modular curves plays an important role in determining 
the possible images of the adelic Galois representations associated to elliptic curves over  $K$, as in Serre's uniformity conjecture \cite{serre-unif-conj, sutherland-Galois-image}
or Mazur's ``Program B'' \cite{rouse-sutherland-zureick-brown}.

Unfortunately, the use of the Chabauty--Coleman method for modular curves is hampered by the fact that expressing a modular curve via a singular plane model is a highly inefficient representation: the size of the coefficients in such a model grow rapidly with the level. The purpose of this paper is to introduce some techniques which can be used to perform Coleman integrals on modular curves without instantiating a projective representation.
We demonstrate our methods by comparing them to the Balakrishnan--Tuitman implementation for three classes of modular curves $X_0(N)$, $X_0^+(N)$, and $X_{ns}^+(N)$ in \S~\ref{section4}.

To say more, let us follow \cite{bbk10} and \cite{bt-coleman} by dividing the problem of computing a path integral $\int_P^Q \omega$ into two steps.
\begin{enumerate}
    \item Computing integrals of the form $\int_P^{P'} \omega$ where the endpoints are constrained to lie within the same residue disc. Following \cite{bbk10} we call these \emph{tiny integrals}.
    \item Computing integrals of the form $\int_{P}^{Q} \omega$ for some points $P,Q$ lying in prescribed distinct residue discs. We call these  \emph{large integrals}. (We are free to replace $P$ and $Q$ within their residue discs for the sake of convenience, at the expense of creating new tiny integrals.)
\end{enumerate}

Typically, one handles tiny integrals by directly expanding in power series and integrating term by term, and large integrals by applying the change of variables formula for an analytic Frobenius lift; this has the effect of constructing a linear system whose solutions compute the large integrals simultaneously for all differentials representing a basis of de Rham cohomology.

For tiny integrals on modular curves, we use a uniformizer derived from the $j$-invariant to expand in power series around an algebraic point on the curve. As the resulting series coefficients are algebraic numbers, we can recover them by computing over $\CC$ instead of $\overline{\QQ}_p$, which means we can use the complex-analytic description of the modular curve as a quotient of the upper half plane.

For large integrals on modular curves, we replace the Frobenius lift with the $p$-th Hecke operator $T_p$, which has the advantage of being defined \emph{algebraically} on the whole curve.

We conclude by expressing the hope that our present methods (or some improvements thereof) can be adapted to \emph{iterated} Coleman integrals, such as those appearing in the \emph{quadratic Chabauty} method \cite{BD1, BD2}. This is of particular concern because while many modular curves do not obey the bound on the Mordell--Weil rank required for Chabauty--Coleman to apply (e.g., in the nonsplit Cartan case), they do almost always obey the rank bound for quadratic Chabauty \cite{qc-bound}.
Thus a model-free implementation of iterated Coleman integration for modular curves would make it possible to extend the work of Balakrishnan, Dogra, M\"{u}ller, Tuitman and Vonk \cite{qc-mod-curves, cursed-curve} applying quadratic Chabauty to solve the rational points problem on some low-genus nonsplit Cartan modular curves. We discuss some issues around model-free iterated Coleman integrals on modular curves at the end of the paper (\S\ref{section6}).


%

\section{Preliminaries and Notations}\label{section2}

Throughout this paper, we write $\HH$ to be the complex upper half plane,
$\HH^+ := \HH \cup \mathbb{P}^1(\QQ)$, and $\GL_2(\QQ)^+ := \{ \alpha \in \GL_2(\QQ) : \det(\alpha) > 0 \}$. For an integer $N \geq 1$, let $H$ be
any subgroup of $\GL_2(\ZZ/N\ZZ)$ satisfying:

\begin{itemize}\label{H_subgroup}
\item $-I \in H$;
  \item The determinant map $\det: H \rightarrow (\ZZ/N\ZZ)^\times$ is
    surjective.
  \end{itemize}

  These ensure that the resulting modular curves in consideration are smooth, irreducible, geometrically irreducible and defined over $\QQ$. We use the standard notation for well-known congruence subgroups $\Gamma(N), \Gamma_1(N), \Gamma_0(N)$ and $\Gamma_H(N) := \{ \begin{psmallmatrix} a & b \\ c & d \end{psmallmatrix} \in
  \SL_2(\ZZ) : \begin{psmallmatrix} a & b \\ c & d \end{psmallmatrix} \pmod{N} \in H\}$.
  
  \subsection{Modular curves and modular forms}
For any congruence subgroup $\Gamma \subseteq \SL_2(\ZZ)$, we write $Y(\Gamma) :=
\Gamma \backslash \HH$ to be the set of orbits under the fractional linear
transformation and $X(\Gamma):= \Gamma \backslash \HH^+$ to be the
compactification of $Y(\Gamma)$ by adjoining cusps. If $\Gamma = \Gamma(N),
\Gamma_0(N), \Gamma_1(N)$ or $\Gamma_H(N)$, we write its corresponding modular
curve $X(\Gamma)$ as $X(N),X_0(N), X_1(N)$ or $X_H(N)$, and if the level $N$ is clear, we
will write it as $X_*$ as appropriate.

The modular curve $X(\Gamma)$ is a connected compact Riemann surface, i.e., an algebraic curve,
and it has a moduli interpretation: the noncuspidal $\overline{\QQ}$-points on $X(\Gamma)$ correspond to isomorphism classes of pairs $(E,\phi)$ where $E$ is an elliptic curve defined over $\overline{\QQ}$ and $\phi$ is an isomorphism between the $N$-torsion points $E[N]$ and the group $\ZZ/N\ZZ\times \ZZ/N\ZZ$. Here, two pairs $(E_1,\phi_1)$, $(E_2,\phi_2)$ are isomorphic if there is an isomorphism $\psi$ of elliptic curves and some invertible matrix $M \in \Gamma$ such that the following diagram commutes:
\begin{center}
  \begin{tikzcd}
    E_1[N] \arrow[r, "\phi_1"] \arrow[d, "\psi"] & (\ZZ/N\ZZ)^2 \arrow[d,  "M"] \\
    E_2[N] \arrow[r, "\phi_2"]                &  (\ZZ/N\ZZ)^2.                               
  \end{tikzcd}
\end{center}

The absolute Galois group $\Gal (\overline{\QQ}/\QQ)$ acts on the non-cuspidal $\overline{\QQ}$-points on $X_\Gamma$ by $\sigma\cdot (E,\phi) := (E^\sigma,\phi \circ \sigma^{-1}),$ where $\sigma\in \Gal(\overline{\QQ}/\QQ)$ We say that a point $(E,\phi)$ is $\QQ$-rational if it is invariant under the
$\Gal(\overline{\QQ} / \QQ)$-action. A necessary condition for
$(E,\phi)$ to be $\QQ$-rational is that $E$ is defined over $\QQ$.

We denote by $\calM_k(\Gamma)$ (resp. $\calS_k(\Gamma)$) the space of weight
$k$ modular forms (resp. cusp forms) with respect to $\Gamma$. There is a canonical isomorphism of $\CC$-vector spaces between $\calS_2(\Gamma)$ and the space of
degree $1$ holomorphic differentials on $X(\Gamma)$.

\subsection{Hecke operators}
We begin with the definition of Hecke operators as double coset operators on the space of modular forms.

\begin{defn}
Let $\Gamma_1$,$\Gamma_2$ be two congruence subgroups of $\slz$, $\alpha \in \GL_2(\mathbb{Q})^+.$ The double coset operator is defined as:

\begin{align*}
    [\Gamma_1\alpha\Gamma_2]_k : \mathcal{M}_k(\Gamma_1) &\rightarrow \mathcal{M}_k(\Gamma_2) \\
    f &\mapsto f[\Gamma_1\alpha\Gamma_2]_k := \sum f[\beta_i]_k,
\end{align*}

\noindent
where \{$\beta_i$\} is the finite set of double coset representatives of $\Gamma_1\backslash \Gamma_1\alpha\Gamma_2$ and $f[\alpha]_k (\tau) = \det(\alpha)^{k-1} (c \tau + d)^{-k} f(\alpha \cdot \tau)$ is the slash-$k$ operator.
\end{defn}


Let $H$ be a subgroup of $\GL_2(\ZZ/N\ZZ)$ and $\Gamma_H \subseteq \slz$ be the lift of $H$ to $\slz$. Let $p$ be a prime  not dividing $N$. Assaf gives an adelic description of the Hecke operator $T_p$ as a correspondence \cite{assaf}. They also relate this description to double cosets, which we use as a definition of $T_p$, the Hecke operator at $p$.

\begin{defn}
Let $p \not| N$ be a prime and $\alpha \in M_2(\ZZ)$ be such that $\det(\alpha) = p$ and $\alpha \pmod{N} \in H$. The Hecke operator at $p$ is defined to be the double coset operator $T_p := [\Gamma_H\alpha\Gamma_H]$. This does not depend on the choice of $\alpha$.
\end{defn}

On the other hand, the double coset operators induce a map of divisor groups of modular curves after $\ZZ$-linear extension. When specialized to Hecke operators, they can be represented as
\[T_p(\Gamma_H\tau) := \sum_i \Gamma_H \beta_i(\tau),\]
\noindent where $\{\beta_i\}$ are coset representatives $\Gamma_H\backslash \Gamma_H\alpha\Gamma_H.$

\begin{example}\label{eg:hecke_N0}
Consider the modular curve $X_0(N) := \Gamma_0(N) \backslash \HH^+$. For $p$ not dividing $N$, the Hecke operator $T_p$ can be defined using $\alpha =  \begin{psmallmatrix}
  1 & 0\\
  0 & p
\end{psmallmatrix}$. The coset representatives are given by 
$\{ 
\begin{psmallmatrix}
  1 & i\\
  0 & p
\end{psmallmatrix}: i \in \{0, \ldots, p-1 \} \}\cup
\{\begin{psmallmatrix}
  p & 0\\
  0 & 1
\end{psmallmatrix}
\}$. Recall that a point on the modular curve $X_0(N)$  is a pair $(E,C)$, where $E$ is an elliptic curve and $C$ is a cyclic subgroup of order $N$. The algebraic description of the Hecke operator at $p$ acting on point is given by: \[T_p(E,C) = \sum_{D \subseteq E[p],\mbox{ $|D| = p$}} (E/D, (C + D)/D).\]
\end{example}

More generally, for any subgroup $H \subseteq \GL_2(\ZZ/N\ZZ)$ and $p$ coprime to $N$, we consider the fiber product $X_H(p,N) := X_0(p) \times_{X(1)} X_H(N)$. There are two degeneracy maps $\alpha,\beta\colon X_H(p,N) \rightarrow X_H(N)$ defining the Hecke operator at $p$: one map corresponds to forgetting the cyclic group of order $p$, while the other corresponds to replacing $E$ with its quotient by the cyclic group of order $p$. By Picard functoriality, for a point $(E,\phi) \in X_H$ where the level structure $\phi$ is determined by $H$, we have an algebraic description of the Hecke operator at $p$: \[T_p(E,\phi) := \alpha^* \beta_* (E,\phi) = \sum_{f:E\rightarrow E', \deg(f) = p} (E',f(\phi)).\]

\subsection{Coleman integrals}\label{sec:coleman_integration}

Coleman's construction of $p$-adic line integrals share many properties as their complex-analytic counterparts. Below we record some properties of Coleman integrals.

\begin{theorem}{\cite{coleman-de-shalit,coleman-ann}} \label{coleman_def}
Let $X/\QQ_p$ be a smooth, projective, and geometrically irreducible curve with good reduction at $p$. Let $\mathfrak{X}/\ZZ_p$ be the smooth model of $X$. 
Then there is a unique way to assign a $p$-adic integral 
\[ \int_P^Q \omega \in \overline{\QQ}_p\]
for every choice of a wide open subspace $U$ of $X^{\mathrm{an}}$, two points $P,Q \in U(\overline{\QQ}_p)$,  and a $1$-form $\omega \in H^0(U,\Omega^1)$ satisfying:
\begin{enumerate}
    \item The integral is $\overline{\QQ}_p$-linear in $\omega$.
    \item We have the following commutativity property:
    \begin{equation*}
        \int_P^Q \omega + \int_{P'}^{Q'} \omega = \int_P^{Q'} \omega + \int_{P'}^Q \omega
    \end{equation*} 
    and so can define $\int_D \omega$ unambiguously for $D \in \Div^0_U(\overline{\QQ}_p)$
    so that
    \begin{equation*}
        \int_{D+D'} \omega = \int_D \omega + \int_{D'} \omega, \qquad \int_D \omega = \int_P^Q \omega \mbox{ for $D = (Q)-(P)$.}
    \end{equation*}
    In particular we have additivity in the endpoints:
        \begin{equation*}
           \int_P^Q \omega = \int_P^R \omega + \int_R^Q \omega.
        \end{equation*}
    
    \item For $P$ and $Q$ in the same residue disc, $\int_P^Q \omega$ can be computed in terms of power series in some uniformizer by formal term-by-term integration (see below).
        \item The integral is compatible with the action of $\Gal(\overline{\QQ}_p/\QQ_p)$. In particular, if $P,Q \in X(\QQ_p)$, then $\int_P^Q \omega \in \QQ_p$.
        
        \item Let $P_0 \in X(\overline{\QQ}_p)$ be fixed and $\omega \in H^0(X, \Omega^1)$ be nonzero. Then for any $x \in \mathfrak{X}(\overline{\FF}_p)$, the set of $P \in X(\overline{\QQ}_p)$ reducing to $x$ such that $\int_{P_0}^P \omega = 0$ is finite.
        
        \item The change of variables formula holds: if $U \subseteq X, V \subseteq Y $ are wide open subspaces of the curves $X,Y$, $\omega$ is a $1$-form on $V$, and $\phi\colon U \rightarrow V$ is a rigid analytic map, then   
        \begin{equation*}
            \int_P^Q \phi^* \omega = \int_{\phi(P)}^{\phi(Q)} \omega.
        \end{equation*}
        In particular, restriction of $\omega$ from $U$ to a wide open subspace does not change $\int_P^Q \omega$.
        
        \item The Fundamental Theorem of Calculus holds: for any rigid analytic function $f \in H^0(U, \mathcal{O})$, $\int_P^Q df = f(Q) - f(P)$.
\end{enumerate}
\end{theorem}

\begin{remark} \label{remark:change of variables for correspondences}
    Thanks to additivity in the endpoints, we may upgrade the change of variables formula to apply to correspondences of curves, not just maps. We will apply this to the Hecke operator $T_p$ for modular curves in \S\ref{sec:colemantotiny}.
\end{remark}

\begin{defn}\label{def:tiny_integral}
The Coleman integral $\int_P^Q \omega$ is called a \emph{tiny integral} if $P$ and $Q$ reduce to the same point in $\mathfrak{X}(\overline{\FF}_p)$, i.e., they lie in the same residue disc.

To compute a tiny integral $\int_P^Q \omega$, we may first express $\omega$ as a power series in terms of a uniformizer at some point in the residue disc of $P$ and $Q$, then formally integrate the power series and evaluate at the endpoints: \[\int_P^Q \omega = \int_{u(P)}^{u(Q)} \omega(u) = \int_{u(P)}^{u(Q)}\sum_{i=0}^\infty a_i u^i du= \sum_{i=0}^\infty \frac{a_i}{i+1} (u(Q)^{i+1} - u(P)^{i+1}).\]
\end{defn}

Coleman’s construction is suitable for explicit computation. In \cite{bbk10}, the authors demonstrated an algorithm to compute single Coleman integrals for hyperelliptic curves. Their method is based on Kedlaya's algorithm for computing the Frobenius action on the de Rham cohomology of hyperelliptic curves \cite{kedlaya-mw}; this approach has been generalized to arbitrary smooth curves \cite{bt-coleman, tuit16, tuit17}. Current implementations of this approach require nice affine plane models for the curves as inputs. Since modular curves tend to have large gonality \cite{large-gonality}, such models are not readily available and are often bottlenecks in existing algorithms.
\section{Coleman Integration on Modular Curves}\label{section3}

In this section, we introduce an algorithm that computes single Coleman integrals of everywhere holomorphic differentials between any two points on modular curves. The modular curves in consideration are of the form $X_H(N)$, where $H \subseteq \GL_2(\ZZ/N\ZZ)$ satisfies the conditions from the earlier section, or $X_H^+(N)$, where we quotient the modular curves by an Atkin-Lehner involution. An innovation is that the algorithm does not require affine models of the modular curves, which are often required in known algorithms.

Given a modular curve $X$ of genus $g$ over $\QQ$ corresponding to the congruence subgroup $\Gamma$, a holomorphic differential $1$-form $\omega$, and two $p$-adic points $Q,R \in X(\QQ_p)$, our algorithm to compute $\int_R^Q \omega$ consists of the following major steps:

\begin{enumerate}

\item (Reduction) Express $\int_R^Q \omega$ in terms of the Hecke action and tiny integrals.

\item (Basis and uniformizer) Find a basis of holomorphic $1$-forms and a suitable uniformizer.

\item (Hecke action) Compute the action of Hecke operators on the space of cusp forms and divisors.

\item (Power series expansion) Express the $1$-forms as a power series expansion in the chosen uniformizer.

\item (Evaluation) Formally integrate the tiny integrals and evaluate at the endpoints.
  \end{enumerate}

\subsection{Coleman integrals as sums of tiny integrals}\label{sec:colemantotiny}
Let $p$ be a prime of good reduction on $X$. Let $Q,\,R$ be two points in $X(\QQ_p)$ and choose a $\QQ$-basis $\{\omega_1,\ldots,\omega_g\}$ of $H^0(X,\Omega^1)$.  The Hecke operator $T_p$ acts on $\mathcal{S}_2(\Gamma)$, corresponding to the holomorphic $1$-forms. Denote by $A$ the Hecke matrix with respect to a basis of $\mathcal{S}_2(\Gamma)$. By linearity, we have:

\[
T_p^*\begin{pmatrix} \omega_1 \\\vdots \\ \omega_g \end{pmatrix}  = A\begin{pmatrix} \omega_1 \\\vdots \\ \omega_g \end{pmatrix} \implies \begin{pmatrix} \int^Q_RT_p^*\omega_1 \\\vdots \\ \int^Q_RT_p^*\omega_g \end{pmatrix}  = A\begin{pmatrix} \int^Q_R\omega_1 \\\vdots \\ \int^Q_R\omega_g \end{pmatrix}.
\]

For any $\omega \in H^0(X,\Omega^1)$, the functoriality property of Coleman integrals and the definition of Hecke operators as correspondences imply that  
\[\int^Q_R T_p^*(\omega) = \int^{T_p(Q)}_{T_p(R)} \omega = \sum_{i=0}^{p} \int^{Q_i}_{R_i} \omega,\] where $T_p(Q) = \sum_{i=0}^p Q_i$ and  $T_p(R) = \sum_{i=0}^p R_i$. The summands are elliptic curves over $\overline{\QQ}_p$ that are $p$-isogenous to $Q$ and $R$ respectively.

By considering the expression $(p+1)\int_R^Q \omega - \int_R^Q T_p^* \omega$, we obtain the following system of equations:

\begin{equation}\label{eq:fundamental-eqn}
    ((p+1)I-A)\begin{pmatrix} \int^Q_R\omega_1 \\\vdots \\ \int^Q_R\omega_g \end{pmatrix} = \begin{pmatrix} \sum_{i=0}^{p}\int^Q_{Q_i} \omega_1 - \sum_{i=0}^{p}\int^R_{R_i} \omega_1 \\\vdots \\ \sum_{i=0}^{p}\int^Q_{Q_i} \omega_g - \sum_{i=0}^{p}\int^R_{R_i} \omega_g \end{pmatrix}.
\end{equation}

\begin{proposition} In Equation \eqref{eq:fundamental-eqn}, we have the following:

\begin{enumerate}
\item Each $Q_i$ (resp. $R_i$) is in the same residue disc as $Q$ (resp. $R$).
    \item The matrix $(p+1)I - A$ is invertible.
\end{enumerate}
\end{proposition}

\begin{proof}
     Let $j$ be the $j$-invariant function. The $Q_i$'s are $p$-isogenous to $Q$ and therefore $j(Q_i)$ are the roots of $\Phi_p(j(Q),X) = 0$, where $\Phi_p(X,Y)$ is the modular polynomial of level $p$. By the Eichler-Shimura relation \cite[Theorem 8.7.2]{ds}, $\Phi_p(X,Y) = (X^p - Y)(X - Y^p) \pmod{p\ZZ[X,Y]}$, so the endpoints of the summands are in the same residue disc. The Ramanujan bound on the Hecke eigenvalues $|a_p| \leq 2 \sqrt{p}$ implies that each eigenvalue of $(p+1)I - A$ has absolute value at least $p+1-2\sqrt{p} = (\sqrt{p}-1)^2 > 0$, so $(p+1)I-A$ is invertible.
\end{proof}

Therefore, we
have reduced the problem of computing Coleman integrals between any two arbitrary endpoints to computing tiny integrals of the
form $\int^Q_{Q_i} \omega$.

\begin{remark}  \label{remark:no exact term}
    We pause to note two important differences between the form of the equation \eqref{eq:fundamental-eqn} and the corresponding equation \cite[(7)]{bbk10}. One is that we work with a basis of $H^0(X, \Omega^1)$ rather than a full basis of de Rham cohomology (which has twice the dimension). The other is that since $T_p$ is defined on all of $X$ rather than on a wide open subspace, it pulls holomorphic forms back to holomorphic forms; in other words, $T_p$ acts on $H^0(X, \Omega^1)$ itself and not merely on de Rham cohomology. Consequently, there is no need to correct \eqref{eq:fundamental-eqn} with terms coming from the Fundamental Theorem of Calculus.
\end{remark}

  \subsection{A basis of \texorpdfstring{$H^0(X,\Omega^1)$}{H0XOmega1}}\label{zywina-basis}

  Computing the space of cusp forms for congruence subgroups $\Gamma(N), \Gamma_1(N)$ and $\Gamma_0(N)$ is well-studied and implemented in various software packages \cite{mf-stein,computingmf,sagemath}.
  
  For the congruence subgroups $H \subseteq \GL_2(\ZZ/N\ZZ)$, we adopt Zywina's approach \cite{zywina}. There is an isomorphism of $\CC$-vector spaces $\mathcal{S}_2(\Gamma(N),\QQ(\zeta_N))^H \cong H^0(X_H, \Omega_{X_H})$, where $\zeta_N$ is a $N$-th root of unity. It remains to specify the action of $\GL_2(\ZZ/N\ZZ)$ on $\mathcal{S}_2(\Gamma(N),\QQ(\zeta_N))$ \cite{zywina, BN20}. Recall that $\SL_2(\ZZ)$ is freely generated by the matrices $S := \begin{psmallmatrix}
0 & -1 \\ 1 & 0 
\end{psmallmatrix}$ and $ T := \begin{psmallmatrix}
1 & 1 \\ 0 & 1 \end{psmallmatrix}$.  Since cusp forms of $\mathcal{S}_2(\Gamma(N))$ have $q_N$-expansions for $q_N := e^{2\pi i /N}$, the slash-$k$ operator by $T$ introduces a factor of $\zeta_N^n$ into the $n$-th Fourier coefficient. On the other hand, the action by $S$ is given by a linear combination of the basis of cusp forms on $\Gamma(N)$ where the coefficients depend on a certain Atkin-Lehner operator $w_N$. Since $\GL_2(\ZZ/N\ZZ)/\SL_2(\ZZ/N\ZZ) \xrightarrow{\cong} (\ZZ/N\ZZ)^\times$, there is an action of $\SL_2(\ZZ/N\ZZ)$ induced from $\SL_2(\ZZ)$ on the cusp forms and the matrix $\begin{psmallmatrix}
        1 & 0 \\ 0 & d
    \end{psmallmatrix}$ acts on the coefficients of the $q_N$-expansion by $\zeta_N \mapsto \zeta_N^d$. One proceeds in a similar manner for Atkin-Lehner quotients of modular curves.
\subsection{Hecke operators as double coset operators}\label{sec:hecke_operator}

Hecke operators act on both cusp forms and points on the modular curve. To compute them as a double coset operator, we need the coset representatives of $\Gamma_H \backslash \Gamma_H \alpha \Gamma_H$. A few key lemmas will sketch a procedure to compute these representatives.

\begin{lemma}{(\cite{ds} Lemmata 5.1.1, 5.1.2)}\label{lemma:coset_rep}
Let $\Gamma, \Gamma_1, \Gamma_2$ be congruence subgroups and  $\alpha \in \GL_2^+(\QQ)$. Then,

\begin{enumerate}
    \item $\alpha^{-1} \Gamma \alpha \cap \SL_2(\ZZ) \leq \SL_2(\ZZ)$ is a congruence subgroup.
    \item There is a bijection:
    \begin{align*}
        (\alpha^{-1} \Gamma_1 \alpha \cap \Gamma_2 )\backslash \Gamma_2 &\leftrightarrow \Gamma_1 \backslash \Gamma_1 \alpha \Gamma_2 \\
         (\alpha^{-1} \Gamma_1 \alpha \cap \Gamma_2 )\gamma_2 &\mapsto \Gamma_1 \alpha \gamma_2
    \end{align*}

\noindent    More precisely, $\{\gamma_{2,i}\}$ is a set of coset representatives for $(\alpha^{-1} \Gamma_1 \alpha \cap \Gamma_2 )\backslash \Gamma_2$ if and only if $\{\alpha \gamma_{2,i}\}$ is a set of coset representatives of $\Gamma_1 \backslash \Gamma_1 \alpha \Gamma_2$.
\end{enumerate}
\end{lemma}

\begin{lemma}{(\cite{shimura} Lemma 3.29(5))} \label{lemma:shimura_coset}
Let $\alpha \in M_2(\ZZ)$ be such that $\det(\alpha) = p$ and $\alpha \pmod{N} \in H$. If $\Gamma_H \alpha \Gamma_H = \bigcup_i \Gamma_H \alpha_i$ is a disjoint union, then $\SL_2(\ZZ) \alpha \SL_2(\ZZ) = \bigcup_i \SL_2(\ZZ)\alpha_i$ is a disjoint union.
\end{lemma}

The procedure for computing the Hecke operator as a double coset operator is as follows:

\begin{enumerate}
    \item Find $\alpha \in M_2(\ZZ)$ satisfying $\det(\alpha) = p$, $\alpha \pmod{N} \in H$.
    
    \item Find the coset representatives $\{\alpha_i\}$ in $(\alpha^{-1} \SL_2(\ZZ) \alpha \cap  \SL_2(\ZZ))\backslash \SL_2(\ZZ)$. Usually, $\alpha$ will be chosen such that $(\alpha^{-1} \SL_2(\ZZ) \alpha \cap  \SL_2(\ZZ))$ has a simple description. By Lemma \ref{lemma:coset_rep}, $\SL_2(\ZZ)\backslash \SL_2(\ZZ) \alpha \SL_2(\ZZ)$ has coset representatives $\{\alpha \alpha_i\}$.
    
    \item By Lemma \ref{lemma:shimura_coset}, for each $\alpha \alpha_i$, find $\beta_i \in \SL_2(\ZZ)$ such that $\beta_i \alpha \alpha_i \in \Gamma_H$. Then $\{ \beta_i \alpha \alpha_i\}$ will be the desired coset representatives for $\Gamma_H \backslash \Gamma_H \alpha \Gamma_H$.
\end{enumerate}

In most cases, our choice of uniformizer at $Q$ will be $j- j(Q)$, where $j$ is the $j$-invariant function. The Hecke operator acts on $\Div(X)$ by sending a point to a sum of points that are $p$-isogenous to it, whose $j$-invariants are the roots of $\Phi_p(j(Q),X) = 0$. A table of small modular polynomials is available in \cite{bls10, bos16}.

\subsection{Tiny integrals from complex approximation}

Theorem \ref{coleman_def} implies that $\int_{Q_i}^Q \omega = \int_{q(Q_i)}^{q(Q)} f(q) dq/q$, where $f(q)$ is the Fourier expansion of the corresponding cusp form, can be expressed as a power series in terms of a uniformizer. 
We compute Taylor coefficients of the cusp forms and uniformizer around a point, and recover the power series coefficients as algebraic approximations of the complex solutions of a system of equations. The algebraic approximations are done using standard lattice-reduction libraries (e.g., \texttt{algdep} from \texttt{PARI/GP}).

\begin{algorithm}
Computing $\sum_{i=0}^{p}\int^Q_{Q_i} \omega$. \label{alg:tiny_integral}

\textbf{Input:}
\begin{itemize}
    \item $\tau_0 \in \HH$ such that $\Gamma\tau_0$ corresponds to a rational point $Q$ on $X$, and $q_0 := e^{2\pi i \tau_0/h}$ where $h$ is the width of the cusp.
    \item A good prime $p$ which does not divide $j(Q)$ or $j(Q)-1728$. 
    \item A cusp form $f\in \mathcal{S}_2(\Gamma)$ specified via its $q$-expansion where $q = e^{2\pi i \tau/h}$. We denote the corresponding $1$-form by $\omega$.

\end{itemize}

\textbf{Output:}
\begin{itemize}
    \item The sum of tiny Coleman integrals $\sum^p_{i=0}\int^Q_{Q_i} \omega \in \QQ_p$, where $T_p(Q) = \sum_{i=0}^p Q_i$.
\end{itemize}

\textbf{Steps:}
\begin{enumerate}
\item[1.] \label{algstep:tiny_1} (Writing $\omega$ as a power series in terms of a uniformizer $u$) Fix a precision $n$. Find $x_i \in \QQ$ such that
\begin{align} \label{eq:omega_j_exp}
    \omega = \left(\sum_{i=0}^n x_i(u)^n + O(u^{n+1})\right)d(u).
\end{align}

\begin{enumerate}
    \item[a.] Perform Taylor expansions of $u$ and $\omega_i$ at $q=q_0$:
    \begin{align*}
    & \omega = \sum_{i=0}^{C_1} b_i(q-q_0)^i + O((q-q_0)^{C_1+1})dq,\\
    & u = \sum_{i=1}^{C_2} a_i(q-q_0)^i + O((q-q_0)^{C_2+1}),\\
    & d(u) = \left(\sum_{i=1}^{C_2} ia_i(q-q_0)^{i-1} + O((q-q_0)^{C_2})\right)dq,
\end{align*}
    where $C_1,C_2$ are some fixed precision determined by $n$ and the norm of $q_0$. The coefficients $a_i,\,b_i$ are in $\CC$.
    \item[b.] Substitute the Taylor expansions of $\omega,\,u$ and $d(u)$ at $q = q_0$ in equation (\ref{eq:omega_j_exp}). Comparing the coefficients of $(q-q_0)^k$ on both sides yields the following linear system:
    \[
\begin{pmatrix}
    a_1       & 0 & 0 & \dots & 0 \\
    2a_2       & a_1^2 & 0 & \dots & 0 \\
    3a_3       & 3a_1a_2 & a_1^3 & \dots & 0\\
    \vdots       & \vdots & \vdots & \ddots & \vdots    \\
    (n+1)a_{n+1}      & \sum_{i=1}^{n}a_i(n+1-i)a_{n+1-i} &\ldots & \dots & a_1^{n+1}
\end{pmatrix} \cdot
\begin{pmatrix}
    x_0      \\
    x_1      \\
    x_2     \\
    \vdots   \\
    x_n
\end{pmatrix} = 
\begin{pmatrix}
    b_0   \\
    b_1     \\
    b_2  \\
    
    \vdots   \\
    b_n
\end{pmatrix}
\]
    \item[c.] The solutions $\textbf{x} = (x_0 \ldots x_n)^T$ of this system of equations can be recovered as elements in $\QQ$ using \texttt{algdep} from \texttt{PARI/GP}.

\end{enumerate}

\item[2.] \label{algstep:tiny_2} Compute $u(Q_i)$ from transcendental and algebraic relations.

\item[3.] \label{algstep:tiny_3} Compute the sum of tiny integrals $\sum\limits_{i=0}^p\int_{Q}^{Q_i}\omega \approx \sum\limits_{i=0}^p \int^{u(Q_i)}_0 (\sum_{j=0}^n x_j u^j du)$ with its $p$-adic expansion.
\end{enumerate}
\end{algorithm}

\begin{remark}
In Step (2) of Algorithm \ref{alg:tiny_integral}, we often use the $j$-invariant function as a uniformizer. We calculate $j(Q_i)$ transcendentally by evaluating the $j(\tau_{Q_i})$, where $\tau_{Q_i} \in \HH$ is the corresponding point in the fundamental domain of $\Gamma$, and then obtain an algebraic approximation. On the other hand, the roots of the modular polynomial $\Phi_p(x,j(Q)) = 0$ are the $j$-invariants of elliptic curves that are $p$-isogenous to $Q$. This gives another algebraic method to compute $j(Q_i)$ and check for correctness.

\end{remark}

\section{Examples}\label{section4}

In this section, we will demonstrate the model-free Coleman integration algorithm for three classes of examples $X_0(N)$, $X_0^+(N)$, and $X_{\ns}^+(N)$, while gathering the necessary ingredients such as known rational points, basis of differentials and the action of the Hecke operators.

\subsection{The modular curve \texorpdfstring{$X_0(N)$}{X0N}}\label{sec:X_0_N}

The modular curve $X = X_0(N)$ is defined as the quotient of the upper half plane by the congruence subgroup $\Gamma_0(N) = \{ \begin{psmallmatrix} \ast & \ast \\ 0 & \ast \end{psmallmatrix} \pmod{N} \} \subseteq \SL_2(\ZZ)$. As a moduli space, the $\QQ$-rational points of $X$ correspond to pairs $(E,\phi)$ where $E$ is an elliptic curve defined over $\QQ$ and $\phi:E\rightarrow E'$ is a $\QQ$-rational isogeny of degree $N$. To search for rational points on $X$, one could start with a naive box search if a model is available. Otherwise, one can still use CM theory to identify all rational points on $X$ arising from cusps or CM elliptic curves which give rise to rational points on $X$; in most cases we expect this to yield all of the rational points. Suppose that only the $j$-invariant of the rational point $Q$ is known. To find the corresponding coset representative on the upper half plane, consider the universal elliptic curve $E$, which provides a construction of an elliptic curve with $j$-invariant $j_E$: \[
y^2 + xy = x^3 - \frac{36}{j_E -1728}x - \frac{1}{j_E - 1728}.
\]

\noindent Substituting $j_E = j(Q)$ yields an elliptic curve $E \cong E_{\tau'} = \CC /(\ZZ + \ZZ \tau')$ such that $j(\tau') = j(Q)$, with $\tau \in \HH$. By iterating through the cosets of $\SL_2(\ZZ)/\Gamma_0(N)$, we can find $\gamma$ such that $j(\gamma \tau') = j( N\gamma \tau') = j(Q)$.

As a result, the point $Q$ corresponds to $\Gamma_0(N) \gamma \tau' \in \Gamma_0(N) \backslash \HH^+$. One can find a basis of weight $2$ cusp forms $\mathcal{S}_2(\Gamma_0(N))$ and the action of Hecke operators on the basis of cusp forms using well-known methods that are implemented in \SageMath{} \cite{mf-stein,sagemath}. For $Q \in X(\QQ)$, let $j(\tau)-j(Q)$ be the chosen uniformizer.

\subsubsection{Example: $X_0(37)$}\label{sec:eg 37}

\begin{itemize}
\item \textbf{Curve data: } We consider the modular curve $X = X_0(37)$. $X$ is a hyperelliptic curve of genus $2$. Comparing relations between $q$-expansions of rational functions $x,\,y \in \CC(X)$, we obtain a plane model $$y^2 = - x^6 - 9x^4 - 11x^2 + 37.$$ It is known that there are four $\QQ$-rational points $Q = (1,-4),\,R = (-1,-4),\, S=(1,4),\,T = (-1,4)$, where $Q,R$ are noncuspidal rational points and $S,T$ are cuspidal rational points \cite{maswd}.

\item \textbf{Rational points:}  Since the $j$-function is a
modular function on $X_0(37)$ and $X_0(37)$ is hyperelliptic, we can
express the $j$-function as a rational function of $x$ and $y$ to compute
that
\begin{align*}
  j(Q) &= -9317 = - 7 \cdot 11^3, \\
  j(R) &= -162677523113838677= - 7 \cdot 137^3 \cdot 2083^3.
\end{align*}

The points $Q, R$ correspond to the elliptic curves $E_Q, E_R$ with
$j$-invariants $j(Q), j(R)$ containing cyclic subgroups of order $37$ (or equivalently, degree $37$-isogenies); this  information can be verified in the L-Function and Modular Forms Database (LMFDB) \cite{lmfdb}. Following the method presented in \S \ref{sec:X_0_N}, we obtain the upper half-plane representatives of $Q,R$: \begin{align*} \tau_Q &\approx 0.5 + 0.17047019819380\cdot i \in \HH , \\ \tau_R &\approx 0.5 + 0.39635999889406\cdot i \in \HH. \end{align*}

\item \textbf{Basis of differential forms:} $\mathcal{S}_2(\Gamma_0(37))$ has $\CC$-dimension $2$. A basis of the space of weight $2$ cusp forms is available on $\SageMath$. Linear algebra yields an eigenbasis $\{f_0,\,f_1\}$ with the following $q$-expansions:
\begin{align*} f_0 &= q+q^3-2q^4+O(q^6), \\ f_1 &= q-2q^2-3q^3+2q^4-2q^5+O(q^6) .\end{align*}

\item \textbf{Hecke action:} We choose $p=3$.  One could compute that $T_3 = \begin{psmallmatrix}
1&0\\
0&-3
\end{psmallmatrix}$ on this eigenbasis. Furthermore, we find $j(Q_i), j(R_i)$ as roots of the modular polynomials $\Phi_3(j(Q),X) = 0$, $\Phi_3(j(R),X) = 0$ where $\Phi_3(X,Y)$ is the modular polynomial of level $3$. 

\item \textbf{Algorithm \ref{alg:tiny_integral} and results:} To maintain consistency with the hyperelliptic model, set:
\begin{align*}
    \omega_0 &:= -\frac{1}{2}f_0 \frac{dq}{q} = \frac{dx}{y}, \\
    \omega_1 &:= -\frac{1}{2}f_1 \frac{dq}{q} = \frac{xdx}{y}.
\end{align*}


By comparing complex coefficients and using $\texttt{algdep}$ to algebraically approximate complex numbers, we obtain rational coefficients $x_i$ in the expansion of $\omega_1$ about $j=j(Q)$:
\begin{align*} \omega_1  = &(-9317) + \frac{717409}{2 \cdot 37 \cdot 47}(j-j(Q))
                             + \frac{253086749261192}{37^2 \cdot 47^3}(j-j(Q))^2
  \\ &+ \frac{176804544077038351043955}{37^3 \cdot 47^5}(j-j(Q))^3 +
       O((j-j(Q))^4) \ \ d(j-j(Q)). \end{align*}
     After that, we substitute the roots into a sum of local power series: 
\begin{align*}
    \sum_{i=0}^3 \int_{Q_i}^Q \omega_1 = &\sum_{i=0}^3 \int_{j(Q_i)-j(Q)}^{0} (-9317) + \frac{717409}{2 \cdot 37 \cdot 47}t + \frac{253086749261192}{37^2 \cdot 47^3}t^2 \\&+ \frac{176804544077038351043955}{37^3 \cdot 47^5}t^3 +  \cdots  dt.
\end{align*}

One could repeat the same processes for $\omega_0$ and $j = j(R)$. The results are listed in the table below; they agree with the results of the $\Magma$ implementation from \cite{balatuit}.
\end{itemize}

\begin{table}[h]

    \centering

    \begin{tabular}{|l|l|}
    \hline
    \rule{0pt}{4ex}

$\sum_{i=0}^{3}\int^Q_{Q_i} \omega_0 $   &$O(3^{14}) $
            \rule{0pt}{4ex} \\
\hline
            \rule{0pt}{4ex}
        $\sum_{i=0}^{3}\int^Q_{Q_i} \omega_1 $  & $3^{2} + 3^{3} + 3^{9} + 3^{10} + 2\cdot 3^{11} +  3^{12} + 2\cdot 3^{13}+ O(3^{14})$
            \rule{0pt}{4ex}
\\\hline
            \rule{0pt}{4ex}
       $\sum_{i=0}^{3}\int^R_{R_i} \omega_0 $  &$O(3^{14}) $    \rule{0pt}{4ex}    
\\\hline
           \rule{0pt}{4ex}    
        $\sum_{i=0}^{3}\int^R_{R_i} \omega_1 $ &$3^{2} + 3^{3} + 3^{9} + 3^{10} + 2\cdot 3^{11} +  3^{12} + 2\cdot 3^{13}+ O(3^{14})$      \rule{0pt}{4ex}    
\\\hline
        
    \end{tabular}
    \caption{Coleman Integrations on $X_0(37)$ as in \S~\ref{sec:eg 37}.}
    \label{table:X_0_37_results}
\end{table}

\subsection{The Atkin-Lehner quotient \texorpdfstring{$X_0^+(N)$}{X0+N}}\label{sec:X_0_N+}

Consider the modular curve $X_0(N)$ from the previous section. There is an Atkin-Lehner involution $w_N := \frac{1}{\sqrt{N}} \begin{psmallmatrix} 0 & -1 \\ N & 0 \end{psmallmatrix}$ acting on $X_0(N)$. One could verify that $w_N^2$ acts as the identity on the $\Gamma_0(N)$-orbits of $\HH$. Let $\Gamma_0^+(N) := \Gamma_0(N) \cup w_N \Gamma_0(N)$. The compactification of the quotient of the upper half plane by $\Gamma_0^+(N)$ gives rise to the modular curve $X:=X_0^+(N)$.

\begin{proposition}\label{moduli_Nplus}
Suppose $\Gamma_0(N) \tau \in X_0(N)$ corresponds to the elliptic curve with torsion data $(E_1, \phi:E_1\rightarrow E_2)$, then $w_N(\Gamma_0(N) \tau )$  corresponds to $(E_2,\hat{\phi}\colon E_2\rightarrow E_1)$, where $\hat{\phi}$ is the dual isogeny of $\phi$.
\end{proposition}

\begin{proof}
$\Gamma_0(N)\tau$ corresponds to $[E_\tau,\langle\frac{1}{N},\tau \rangle]$ up to isomorphism. As $w_N \cdot \tau = \frac{-1}{N\tau}$, $w_N( \Gamma_0(N)\tau)$ corresponds to $[E_{\frac{1}{N\tau}}, \langle\frac{1}{N}, \frac{1}{N\tau} \rangle]$. There is an isomorphism between elliptic curves with a cyclic subgroup of order $N$ and complex tori given by:
\[
[E_\tau, \langle\frac{1}{N},\tau\rangle] \cong \mathbb{C}/(\frac{1}{N}\ZZ 
 + \ZZ\tau).
\]

 It is clear that  $\langle\frac{1}{N},\tau\rangle = \tau \langle 1,\frac{1}{N\tau}\rangle$, hence $E_{\frac{1}{N\tau}}$ is isomorphic to $E_\tau/\langle\frac{1}{N},\tau\rangle$. It remains to check that the dual isogeny of $\phi\colon E \rightarrow E_\tau/\langle\frac{1}{N},\tau\rangle$ is indeed the isogeny induced by $E_{\frac{1}{N\tau}}$. This can be checked by first computing the dual isogeny and comparing kernels.
\end{proof}

The above proposition provides a moduli interpretation for $X_0^+(N)$, i.e., the $\QQ$-points correspond to unordered pairs of elliptic curves $\{\phi_1\colon E_1 \rightarrow E_2, \phi_2\colon E_2 \rightarrow E_1\}$ such that $\phi_1$ is an isogeny of degree $N$, and $\phi_2$ is the dual isogeny, with the additional requirement that they are $\Gal(\overline{\QQ}/ \QQ)$-invariant. By CM theory, it is possible that the elliptic curves $E_1,E_2$ or the isogenies $\phi_1,\phi_2$ may not be defined over $\QQ$ but over a quadratic extension of $\QQ$, and in that case the elliptic curves or isogenies are fixed by the nontrivial Galois element of the quadratic extension.

The expected rational points on $X$ correspond to elliptic curves with complex multiplication. Following \cite{Mercuri2016EquationsAR,stark_classnumber1}, we have a list of discriminants of imaginary quadratic number rings with class number one:
\[
\mathcal{D} = \{ -3,\,-4,\,-7,\,-8,\,-11,\,-12,\,-16,\,-19,\,-27,\,-28,\,-43,\,-67,\,-163\}.
\]

Let $E$ be a CM elliptic curve such that its endomorphism ring $\mathcal{O}_E$ has discriminant $\Delta_E \in \mathcal{D}$. Elliptic curves $E$ such that $N$ splits or ramifies in $\mathcal{O}_E$ give rise to rational points on $X$ \cite{Galbraith_1999}. Iterating through the class number one discriminants, we have list of candidates of expected rational points coming from CM elliptic curves. We denote one of the rational points by $Q$.

The endomorphism ring $\mathcal{O}_E$ is an order in an imaginary quadratic field and therefore has a generator $\tau_E$ and we factor the ideal $(N)$ into a product of principal ideals $\mathfrak{m}\bar{\mathfrak{m}}$ in $\mathcal{O}_E$. Write $\mathfrak{m} = (\alpha)$. Since $\alpha \in \mathcal{O}_E$, there exists integers $c,d$ such that $\alpha = c \tau_E + d$. The Euclidean algorithm gives two integers $a,b$ such that $\gamma = \begin{psmallmatrix} a & b \\ c & d \end{psmallmatrix} \in \SL_2(\ZZ)$. In this case, the upper half plane representative is $\tau_Q = \gamma \cdot \tau_E$.

A basis of cusp forms is given by the forms that are fixed under the Atkin-Lehner involution, $\mathcal{S}_2(\Gamma_0^+(N)) = \{ f \in \mathcal{S}_2(\Gamma_0(N))\colon f[w_N]_2 = f\}$. The action of Hecke operators on $X_0^+(N)$ is given by the following lemma:

\begin{lemma}
Let $\alpha \in \GL_2^+(\QQ)$. The coset representatives of $(\alpha^{-1}\Gamma_0^+(N) \alpha \cap \Gamma_0^+(N))\backslash \Gamma_0^+(N)$ are the same as that of  $(\alpha^{-1}\Gamma_0(N) \alpha \cap \Gamma_0(N))\backslash \Gamma_0(N)$.
\end{lemma}

\begin{proof}
Observe that 
\begin{align*}
    \alpha^{-1}\Gamma_0^+(N)\alpha\cap \Gamma_0^+(N) &= \alpha^{-1}(\Gamma_0(N)\cup w_N\Gamma_0(N))\alpha \cap (\Gamma_0(N)\cup w_N\Gamma_0(N)) \\
    &= (\alpha^{-1}\Gamma_0(N)\alpha \cap \Gamma_0(N))\cup (\alpha^{-1}(w_N\Gamma_0(N))\alpha\cap w_N\Gamma_0(N))
\end{align*}

By Lemma \ref{lemma:coset_rep}, one has an explicit description of the double coset representatives of $\Gamma_0(N)\alpha \Gamma_(N)$ and one could show that the two sets of coset representatives above are equal.
\end{proof}

In particular, for a prime $p$, the Hecke operators $T_p$ on $X_0^+(N)$ and $X_0(N)$ coincide as double coset operators: 
\[
( \ \cdot \ ) |_k [ \Gamma_0^+(N) \alpha \Gamma_0^+(N)] = ( \ \cdot \ ) |_k [ \Gamma_0(N) \alpha \Gamma_0(N)]: f \mapsto \sum_i f |_k \beta_i = \sum_{i=0}^{p-1} f |_k \begin{pmatrix} 1 & i \\ 0 & p \end{pmatrix} + f|_k \begin{pmatrix} p & 0 \\ 0 & 1 \end{pmatrix}.
\]

For the uniformizer, we require a combination of modular functions that is invariant under the Atkin-Lehner involution $w_N$. Since $j(w_N \cdot \tau) = j(-1/N\tau) = j(N\tau)$, we can choose $(j + j_N) - (j+j_N)(Q)$ as our uniformizer at $Q$, where $j_N(\tau) := j(N\tau)$. For a given point $Q = \{E_1 \leftrightarrow E_2\}$ and the points $Q_i$ in the same residue disc, the endpoints of the sum of tiny integrals are $j(Q_i) + j(NQ_i)$ where $j(Q_i)$ and $j(NQ_i)$ can be computed as roots of the modular polynomials as in the previous example.

\subsubsection{Example: $X_0^+(67)$} \label{sec: eg 67}

\begin{itemize}

\item \textbf{Curve data:} We consider the modular curve $X = X_0^+(67)$, which is of genus $2$ and hence hyperelliptic. Again, by comparing relations between $q$-expansions of rational functions $x,\,y \in \CC(X)$, we obtain a plane model $y^2 =  x^6 + 2x^5 + x^4 - 2x^3 + 2x^2 - 4x +1$. A quick box search yields two rational points $R = (0,-1),S = (1,1)$ on $X$.

\item \textbf{Uniformizers:} We use $j + j_N$ as the uniformizer since it is a modular function invariant under the Atkin-Lehner involution.

\item \textbf{Rational points:}  For the rational points $R,S$, their upper half plane representatives can be found as follows. $R$ is the pair $\{\phi_1\colon E_1 \rightarrow E_1, \hat{\phi}_1\colon E_1 \rightarrow E_1\}$, with $j(E_1) =2^65^3$.  $E_1/\mathbb{Q}$ has CM by the ring of integers $\mathcal{O}_{K_1}$ with $K_1 (= \mathbb{Q}(\sqrt{-2}))$ and $67$ splits in $\mathcal{O}_{K_1}$. Similarly, $S$ is the pair $\{\phi_2\colon E_2 \rightarrow E_2, \hat{\phi}_2\colon E_2 \rightarrow E_2\}$, with $j(E_2) =   2^{4}3^35^3$. $E_2/\mathbb{Q}$ has CM by the ring of integers $\mathcal{O}_{K_2}$ with $K_2 = \mathbb{Q}(\sqrt{-43})$ and $67$ splits in $\mathcal{O}_{K_2}$. Note that neither $R$ nor $S$ are fixed by the Atkin-Lehner involution, since that would correspond to the case when $67$ is ramified in the respective endomorphism rings.

We have $j(R) = 2^6 5^3, D(R) = -8$, hence $\tau_R = \sqrt{-2}$.
Following the steps described in the previous section, $(67) = (7+3\sqrt{-2})(7-3\sqrt{-2})$ and the Euclidean algorithm gives:
\[
  7 + 3\sqrt{-2} = 7+ 3\cdot \sqrt{-2} \implies \hat{\gamma} = \begin{pmatrix}
    1       & 2 \\
    3       & 7
\end{pmatrix}.
\]
Therefore,
\begin{align*}
 \hat{\tau}_R = \hat{\gamma}\tau_R &= \frac{\sqrt{-2}+2}{3\sqrt{-2}+7} \\ &\approx 0.298507462686567 + 0.0211076651100462\cdot i.
\end{align*}

Similarly, we have $j(S) = 2^4 3^3 5^3, D(S) = -12, \tau_S = \sqrt{-3}$, $(67) =
(8+\sqrt{-3})(8-\sqrt{-3})$ and the Euclidean algorithm gives:
\[
  8 + \sqrt{-3} = 8 + 1\cdot\sqrt{-3} \implies \hat{\gamma} = \begin{pmatrix}
    -1       & -9 \\
    1       & 8
\end{pmatrix}
\]
Therefore,
\begin{align*}
\hat{\tau_S} = \hat{\gamma}\tau_S &=  -\frac{\sqrt{-3} + 9}{\sqrt{-3} + 8} \\ &\approx 1.11940298507463 - 0.0258515045905802\cdot i.
\end{align*}

\item \textbf{Basis of differential forms:} $\mathcal{S}_2(\Gamma_0(67))$ has dimension $5$. One could compute the action of $w_{67}$ on the space and find a $2$-dimensional subspace spanned by cusp forms invariant under the Atkin-Lehner involution using \SageMath{} to get a basis of $H^0(X,\Omega^1)$: \begin{align*} \omega_0 = f_0  \frac{dq}{q} \  &= 2q -3q^2 - 3q^3 + 3q^4 - 6q^5 + O(q^6) \ \frac{dq}{q},\\ \omega_1 = f_1  \frac{dq}{q} \ &= -q^2 + q^3 + 3q^4 + O(q^6) \ \frac{dq}{q}.
\end{align*}

\item \textbf{Hecke action:} Let $p=13$. The Hecke matrix on this basis is given by $T_{13}= \begin{psmallmatrix}
    -7/2 & 15/2 \\ 3/2 & -7/2
\end{psmallmatrix}$. As before, we find the Hecke images of points as roots of modular polynomials at level $13$.

\item \textbf{Algorithm \ref{alg:tiny_integral} and results:} Step 1 of Algorithm \ref{alg:tiny_integral} gives a power series expansion of the differential forms for the uniformizer $j := j + j_N$ (for simplicity, we use this notation). For example, $\omega_0$ at $j=j(R)$ has the following power series expansion:\begin{align*}
    \omega_0 &= \frac{-1}{2^7 \cdot 5^2 \cdot 7^2
} +  \frac{3047}{2^{15} \cdot 5^5 \cdot 7^6
}(j-j(R)) +  \frac{-38946227}{2^{24} \cdot 5^8 \cdot 7^{10}
}(j-j(R))^2 \\ &+ \frac{33888900627}{2^{32} \cdot 5^{10} \cdot 7^{14}
} + \frac{-110823337943341}{2^{42} \cdot 5^{13} \cdot 7^{17}
}(j-j(R))^3 + O((j-j(R))^4) \ \ d(j-j(R)).
\end{align*}

The endpoints $j(R_i) + j(NR_i)$ appearing in the sums of tiny integrals can be computed following the approach outlined in the earlier section. One repeats the process for $\omega_1$ and $j = j(Q)$. Finally, we tabulate the values of the Coleman integrals; since $X$ is hyperelliptic, we can verify our results as in the previous example.

\end{itemize}

\begin{table}[h]

    \centering
    \begin{tabular}{|l|l|}
    \hline
    \rule{0pt}{4ex}    

        $\sum_{i=0}^{3}\int^R_{R_i} \omega_0 $    & $2\cdot 13 + 13^2 + 3\cdot 13^3 + 7\cdot 13^4 + 11\cdot 13^5 + 8\cdot 13^6 + 8\cdot 13^7 + 7\cdot 13^8 + 13^9 +  O(13^{10})$ 
            \rule{0pt}{4ex} \\
\hline
            \rule{0pt}{4ex}
        $\sum_{i=0}^{3}\int^R_{R_i} \omega_1 $  & $11\cdot 13 + 8\cdot 13^2 + 6\cdot 13^3 + 8\cdot 13^4 + 3\cdot 13^5+ 6\cdot 13^6 + 6\cdot 13^7 + 7\cdot 13^8 + 11\cdot 13^9 + O(13^{10}) $
            \rule{0pt}{4ex}
\\\hline
            \rule{0pt}{4ex}
       $\sum_{i=0}^{3}\int^S_{S_i} \omega_0 $ & $10\cdot  13 + 8\cdot 13^2 + 2\cdot 13^5 + 5\cdot 13^6 + 10\cdot 13^7 + 2\cdot 13^8 + 2\cdot 13^9+ O(13^{10}) $   \rule{0pt}{4ex}    
\\\hline
           \rule{0pt}{4ex}    
        $\sum_{i=0}^{3}\int^S_{S_i} \omega_1 $ &  $3\cdot 13 + 7\cdot 13^2 + 2\cdot 13^3 + 10\cdot 13^4 + 8\cdot 13^5+ 5\cdot 13^6 + 8\cdot 13^8 + 10\cdot 13^9+ O(13^{10}) $   \rule{0pt}{4ex}    
\\\hline
        
    \end{tabular}
    \caption{Coleman Integrations on $X_0^+(67)$ as in \S \ref{sec: eg 67}}
    \label{table:X_0^+(67)_results}
\end{table}


\subsection{The normalizer of the nonsplit Cartan \texorpdfstring{$X_{\ns}^+(p)$}{Xns+p}}\label{sec:X_ns+}

For a prime $p$, we first define the nonsplit Cartan subgroup $C_{\ns}$ and its normalizer $C_{\ns}^+$. Let $\{ 1, \alpha \}$ be a $\FF_p$-basis of $\FF_{p^2}$. Suppose that $\alpha$ satisfies a minimal polynomial $X^2 - tX + n \in \FF_p[X]$. For any $\beta = x + y\alpha \in \FF_{p^2}^\times$, there is a multiplication-by-$\beta$ map with respect to the basis $\{ 1, \alpha \}$:
\begin{align*}
i_\alpha\colon \FF_{p^2}^\times &\rightarrow \GL_2(\FF_p) \\
\beta &\mapsto \begin{pmatrix} x & -ny \\ y & x + ty \end{pmatrix}.
\end{align*}

Given this choice of basis, we define the nonsplit Cartan subgroup $C_{\ns}(p) \subseteq \GL_2(\FF_p)$ as the image of $i_\alpha$. The normalizer of the nonsplit Cartan subgroup $C_{\ns}^+(p)$ is the subgroup generated by $C_{\ns}(p)$ and the conjugation map under $i_\alpha$ coming from $\Gal(\FF_{p^2}/\FF_p)$. If $\alpha$ is chosen to be the squareroot of a quadratic nonresidue $\epsilon$ in $\FF_{p^2}$ (so that the minimal polynomial becomes $X^2 - \epsilon$), then we have a simpler description of the normalizer of the nonsplit Cartan subgroup:
\[
 C_{\ns}^+(p) = \langle \begin{pmatrix} x & \epsilon y \\ y & x\end{pmatrix} , \begin{pmatrix} 1 & 0 \\ 0 & -1 \end{pmatrix}\colon (x,y) \in \FF_p^2 \backslash (0,0) \rangle.
 \]
 
 If $\langle \beta \rangle= \FF_{p^2}^\times$, then we can write down the generators of $C_{\ns}^+(p)$.
 
 \begin{example}
 Let $p = 13, \epsilon = \sqrt{7}, \FF_{p^2}^\times = \langle 1 + \epsilon \rangle$. Then 
 \[
 C_{\ns}^+(13) = \langle \begin{pmatrix} 1 & 7 \cdot 1 \\ 1 & 1 \end{pmatrix}, \begin{pmatrix} 1 & 0 \\ 0 & -1 \end{pmatrix} \rangle.
 \]
 \end{example}
 
 The modular curve corresponding to $C_{\ns}^+(p)$ is defined as the compactification of the quotient of the upper half plane by the lift of $C_{\ns}^+(p)$ to a subgroup $\Gamma_{\ns}^+(p) \subseteq \SL_2(\ZZ)$.
 
 Finding a basis of $\mathcal{S}_2(\Gamma_{\ns}^+(p))$ can be done following Zywina's \Magma{} implementation \cite{zywina}. For the purpose of exposition, suppose $\mathcal{S}_2(\Gamma_{\ns}^+(p)) = \{ f_1, \ldots, f_g \}$.
 
 To find the upper half plane representatives of the expected rational points, we follow a similar procedure for $X_0(N)$. First, in the list of class number one discriminants $\mathcal{D}$, the expected points correspond to the discriminants $\Delta$ such that $p$ is inert in the corresponding order $\mathcal{O}_\Delta$ \cite{Mazur77}. Once we have the list of expected points $\{ P_1, \ldots P_r \}$, one can use the same method of inverting the $j$-invariant function to find $\SL_2(\ZZ)$-orbits $\{ \tau_1, \ldots, \tau_r \}$. The cosets of $\SL_2(\ZZ)/\Gamma_{\ns}^+(p)$ allow us to find the correct upper half plane representatives corresponding to $\{ P_1, \ldots P_r \}$. It remains to identify the cosets of $\SL_2(\ZZ)/\Gamma_{\ns}^+(p)$ using the following bijection:
 \begin{align*}
    \SL_2(\ZZ)/\Gamma_{\ns}^+(p) &\rightarrow \SL_2(\ZZ/p\ZZ)/(C_{\ns}^+(p) \cap \SL_2(\ZZ/p\ZZ))\\
      \Gamma_{\ns}^+(p) \gamma &\mapsto (C_{\ns}^+(p) \cap \SL_2(\ZZ/p\ZZ))\bar{\gamma}.
\end{align*}

Once we have obtained coset representatives $\{ \gamma_i \}$ of $\SL_2((\ZZ/p\ZZ)/(C_{\ns}^+(p) \cap \SL_2(\ZZ/p\ZZ))$, we can verify if $\gamma_i \tau$ is a $\QQ$-rational point on $X$ for $\tau \in \{ \tau_1, \ldots, \tau_r\}$ via the canonical embedding by verifying $[f_1(\gamma_i \cdot \tau) : \ldots : f_g(\gamma_i \cdot \tau)] \in \PP^{g-1}(\QQ)$.

On the cusp forms, there are two major steps to computing the Hecke operator: first find the double coset representatives and then decompose these representatives into products on simpler matrices, where simpler methods can be adopted to compute the slash-$k$ operators \cite{zywina,ds}. For the Hecke operator at the prime $\ell$, we have:
\begin{align*}
[\Gamma_{\ns}^+(p)\alpha\Gamma_{\ns}^+(p)]_2 f = \sum f[\alpha_i]_2 , \hspace{5mm}
\end{align*}
\noindent where $\{ \alpha_i \}_{i = 0, \ldots, p}$ are the double coset representatives of $\Gamma_{\ns}^+(p)\backslash\Gamma_{\ns}^+(p)\alpha\Gamma_{\ns}^+(p)$. The representatives have the form $\alpha_i = \epsilon\epsilon'\begin{psmallmatrix} 1 & 0 \\ 0 & \ell \end{psmallmatrix}\beta$ or $\epsilon\epsilon' \beta \begin{psmallmatrix} \ell & 0 \\ 0 & 1 \end{psmallmatrix}$, where $\epsilon,\epsilon' \in \SL_2(\ZZ)$ depend on $\alpha$, and $\beta$ can be derived from the standard cosets of $\Gamma^0(\ell)\backslash \SL_2(\ZZ)$. These particular decompositions allow easy modifications to Zywina's algorithm for computing the slash-$k$ operator on the determinant 1 matrices \cite{zywina}, while the two rightmost matrices can be resolved using techniques from \cite[Ch. 5]{ds}. In the decomposition of the last coset, we make use of the identity $\begin{psmallmatrix}
1 & 0\\
0 & \ell
\end{psmallmatrix}\begin{psmallmatrix}
m\ell &  n\\
N & 1
\end{psmallmatrix} = \begin{psmallmatrix}
m &  n\\
N & \ell
\end{psmallmatrix}\begin{psmallmatrix}
\ell & 0\\
0 & 1
\end{psmallmatrix}$ where $m\ell - nN = 1$.

A different application of Zywina's algorithm outputs a basis $\{f_1, \ldots, f_g \}$ of weight $2$ cusp forms on $\Gamma_{\ns}^+(p)$. The Hecke matrix $T_p$ can be computed by writing $[\Gamma_{\ns}^+(p)\alpha\Gamma_{\ns}^+(p)]_2 f_i $ as a linear combination of the basis elements of $\mathcal{S}_2(\Gamma(p), \QQ(\zeta_p))$.

The Hecke operator on points can be computed transcendentally and algebraically. Each approach has its (dis)advantages: we can evaluate cusp forms on explicit representatives but this will require a closer analysis of the group structure of $C_{\ns}^+(p)$ and high enough complex precision; the roots of the modular polynomials are the $j$-invariants of isogenous points but these polynomials have large coefficients.

\subsubsection{Example: $X_{\ns}^+(13)$}\label{sec:eg 13}

We consider the `cursed curve' $X = X_{\ns}^+(13)$ of genus $3$ \cite{cursed-curve}. Define $C_{\ns}^+(13)$ by choosing the quadratic nonresidue to be $7$ as in the previous example. Let $\Gamma_{\ns}^+(13)$ be the lift of $C_{\ns}^+(13)$ in $\SL_2(\ZZ)$.

\begin{itemize}
\item \textbf{Basis of differential forms:} Using Zywina's \Magma{} implementation \cite{zywina}, we obtain a basis of cusp forms with the following $q$-expansions:
\begin{align*} 
f_0 = &(3\zeta_{13}^{11} + \zeta_{13}^9 + 3\zeta_{13}^8 + \zeta_{13}^7 + \zeta_{13}^6 + 3\zeta_{13}^5 + \zeta_{13}^4 + 3\zeta_{13}^2 + 1)q \\ &+ (-\zeta_{13}^{10} - 2\zeta_{13}^9 - \zeta_{13}^7 - \zeta_{13}^6 - 2\zeta_{13}^4 - \zeta_{13}^3 - 2)q^2 + O(q^3)\\ f_1 = &(4\zeta_{13}^{11} + 2\zeta_{13}^9 + 5\zeta_{13}^8 + 5\zeta_{13}^5 + 2\zeta_{13}^4 + 4\zeta_{13}^2)q \\ &+
        (-3\zeta_{13}^{11} - 5\zeta_{13}^{10} - 4\zeta_{13}^9 - 4\zeta_{13}^8 - 4\zeta_{13}^7 - 
        4\zeta_{13}^6 - 4\zeta_{13}^5 - 4\zeta_{13}^4 - 5\zeta_{13}^3 - 3\zeta_{13}^2 - 2)q^2 + O(q^3) \\ f_2 =  &(\zeta_{13}^{10} - 2\zeta_{13}^7 - 2\zeta_{13}^6 + \zeta_{13}^3)q \\ &+ (-\zeta_{13}^{11} - 2\zeta_{13}^{10} - 
        2\zeta_{13}^8 - 2\zeta_{13}^5 - 2\zeta_{13}^3 - \zeta_{13}^2 + 2)q^2 + O(q^3), 
\end{align*} 
where $\zeta_{13}$ is a primitive $13$-th root of unity and $q = e^{\frac{2\pi i\tau}{13}}$.
        
\item \textbf{Curve data:} The method of canonical embedding gives us the following model \cite{Galbraith_1996}:
        \begin{align*}\label{eq:cursed_curve}
\begin{split}
    &X^4 - \frac{7}{12}X^3Y - \frac{37}{30}X^2Y^2 + \frac{37}{30}XY^3 - \frac{3}{10}Y^4 - \frac{61}{60}X^3Z + \frac{41}{15}X^2YZ  \\
    &- \frac{103}{60}XY^2Z+ \frac{19}{60}Y^3Z - \frac{23}{6}X^2Z^2 + \frac{87}{20}XYZ^2 - \frac{14}{15}Y^2Z^2 - \frac{199}{60}XZ^3 \\
    &+ \frac{97}{60}YZ^3 - \frac{11}{15}Z^4 = 0,
\end{split}
\end{align*}
\noindent where $X,\,Y$ and $Z$ correspond to $f_0,\,f_1$ and $f_2$ respectively. The rational points can be found by a box search: \[\{ (\frac{3}{5}:2:1), (-2:2:1), (-2:\frac{-9}{2}:1), (-2: \frac{-7}{3}:1), (\frac{7}{3}:2:1), (\frac{5}{4}: 2:1), (11: \frac{43}{2}:1) \}\].

\item \textbf{Uniformizers:} $\mathcal{S}_2(\Gamma_{\ns}^+(13)) \subseteq \mathcal{S}_2(\Gamma(N), \QQ(\zeta_N))$ so the $j$-function is still a modular function for the normalizer of nonsplit Cartan and can be used as an uniformizer.

\item \textbf{Rational points: } Among the class number one discriminants $\Delta$ in $\mathcal{D}$, we find $\Delta$ such that $13$ is inert in the corresponding order $\mathcal{O}_\Delta$. The set $\{ -7,-8,-11,-19,-28,-67,-163 \}$ contains discriminants that give rise to seven expected rational points on $X$. We pick $Q$ to be the point that corresponds to discriminant $-7$, and $R$ to be the point that corresponds to discriminant $-11$. Following the notations in the previous section, we have $\tau_7 = \frac{1}{2} + \frac{1}{2}\sqrt{-7}$ and $\tau_{11} = \frac{1}{2} + \frac{1}{2}\sqrt{-11}$. We then compute the coset representatives of $\SL_2(\ZZ)/\Gamma_{\ns}^+(13)$:
\[
\{g_0,\ldots,g_{77}\} = \{T^i,\, (T^2)ST^i,\, (T^3)ST^i,\,(T^4)ST^i,\,(T^5)ST^i,\,(T^{12})ST^i \mbox{ for } i = 0,\ldots,12\},
\] 
where $T =\begin{psmallmatrix}
1 & 1\\
0 & 1
\end{psmallmatrix} ,\,S=\begin{psmallmatrix}
0 & -1\\
1 & 0
\end{psmallmatrix}$ are the two generators of $\SL_2(\ZZ)$. By evaluating $f_0,\,f_1,\,f_2$ at $g_i(\tau_7)$ and $g_i(\tau_{11})$ for $i \in \{0,\ldots,77\}$, we obtain the correct $\Gamma_{\ns}^+(13)$-orbit representatives for $Q$ and $R$, $\tau_Q = \frac{4 + 2\sqrt{-7}}{3 + \sqrt{-7}}, \tau_R = \frac{13 + \sqrt{-11}}{2}$. As in the previous section, the correct representative for $Q$ can be found by evaluating $\frac{f_0(g_i(\tau_7))}{f_2(g_i(\tau_7))}$ and $\frac{f_1(g_i(\tau_7))}{f_2(g_i(\tau_7))}$ for different coset representatives $g_i$ so that the ratios are rational numbers. Applying the same method to all seven discriminants, we get their corresponding rational points as computed from the model above.

\item \textbf{Hecke action on forms: }  We choose $p=11$. Let $\alpha = \begin{psmallmatrix}
-13 & 44 \\ 42 & -143 
\end{psmallmatrix} \begin{psmallmatrix}
1 & 0 \\ 0 & 11
\end{psmallmatrix}$ be the element $\alpha \in M_2(\ZZ)$ with $\det(\alpha) = 11$, $\alpha \pmod{13} \in C_{\ns}^+(13).$ By Lemmas \ref{lemma:coset_rep} and \ref{lemma:shimura_coset}, the double coset representatives can be found by first considering the coset representatives for $\mathcal{S} := (\alpha^{-1} \slz \alpha \cap \slz) \backslash \slz = \Gamma^0(11) \backslash \slz$. For each $\beta \in \mathcal{S}$, we found a corresponding $\gamma \in \Gamma^0(11)$ such that the representative $\beta' = \gamma \beta \in \Gamma_{\ns}^+(13)$. Denote the set of coset representatives by $\mathcal{S}' := (\alpha^{-1} \Gamma_{\ns}^+(13) \alpha \cap \Gamma_{\ns}^+ (13)) \backslash \Gamma_{\ns}^+(13)$ and the set of corresponding $\gamma$'s by $\Gamma$. Then, we have: \begin{align*} 
\mathcal{S} = \Big\{ &\begin{pmatrix}
1 & i \\ 0 & 1
\end{pmatrix} : i = 0,1, \ldots, 10 \Big\} \cup \Big\{ \begin{pmatrix}
66 & 5 \\ 13 & 1
\end{pmatrix}\Big\}, \\
\Gamma = \Big\{ &\begin{pmatrix}
1 & 0 \\ 0 & 1
\end{pmatrix},
\begin{pmatrix}
1 & 0 \\ -2 & 1
\end{pmatrix}, \begin{pmatrix}
1 & 11 \\ 0 & 1
\end{pmatrix}, \begin{pmatrix}
1 & -55 \\ 0 & 1
\end{pmatrix},\begin{pmatrix}
1 & 22 \\ 0 & 1
\end{pmatrix},\begin{pmatrix}
1 & -44 \\ 0 & 1
\end{pmatrix},\begin{pmatrix}
1 & 33 \\ 0 & 1
\end{pmatrix}, \\ &\begin{pmatrix}
1 & -33 \\ 0 & 1
\end{pmatrix}, \begin{pmatrix}
1 & 44 \\ 0 & 1
\end{pmatrix},\begin{pmatrix}
1 & -22 \\ 0 & 1
\end{pmatrix},\begin{pmatrix}
-1 & -55 \\ 0 & -1
\end{pmatrix}, 
\begin{pmatrix}
1 & -44 \\ 0 & 1
\end{pmatrix}\Big\}, \\
\mathcal{S}' = \Big\{ &\begin{pmatrix}
1 & 0 \\ 0 & 1
\end{pmatrix},\begin{pmatrix}
1 & 1 \\ -2 & 1
\end{pmatrix}, \begin{pmatrix}
1 & 13 \\ 0 & 1
\end{pmatrix},\begin{pmatrix}
1 & -52 \\ 0 & 1
\end{pmatrix},\begin{pmatrix}
1 & 26 \\ 0 & 1
\end{pmatrix},\begin{pmatrix}
1 & -39 \\ 0 & 1
\end{pmatrix},\begin{pmatrix}
1 & 39 \\ 0 & 1
\end{pmatrix}, \\ &\begin{pmatrix}
1 & -26 \\ 0 & 1
\end{pmatrix},\begin{pmatrix}
1 & 52 \\ 0 & 1
\end{pmatrix},\begin{pmatrix}
1 & -13 \\ 0 & 1
\end{pmatrix},\begin{pmatrix}
-1 & -65 \\ 0 & -1
\end{pmatrix},\begin{pmatrix}
-506 & -39 \\ 13 & 1
\end{pmatrix}\Big\}.
\end{align*}

From the bijection
\begin{align*} \Gamma_{\ns}^+(13)\backslash \Gamma_{\ns}^+(13) \alpha \Gamma_{\ns}^+(13) &\rightarrow (\alpha^{-1} \Gamma_{\ns}^+(13) \alpha \cap \Gamma_{\ns}^+(13)) \backslash \Gamma_{\ns}^+(13) \\ \Gamma_{\ns}^+(13) \delta &\mapsto (\alpha^{-1} \Gamma_{\ns}^+(13) \alpha \cap \Gamma_{\ns}^+(13)) \alpha^{-1} \delta \end{align*}
\noindent we can get the double coset representatives of $\Gamma_{\ns}^+(13)\backslash \Gamma_{\ns}^+(13) \alpha \Gamma_{\ns}^+(13)$: \begin{align*}
\mathcal{S}_\alpha = \Big\{
&\begin{pmatrix}
-13 & 4 \\ 42 & -13
\end{pmatrix}
\begin{pmatrix}
1 & 0 \\ 0 & 11
\end{pmatrix}
\begin{pmatrix}
1 & 0 \\ 0 & 1
\end{pmatrix}
\begin{pmatrix}
1 & 0 \\ 0 & 1
\end{pmatrix},
\begin{pmatrix}
-13 & 4 \\ 42 & -13
\end{pmatrix}
\begin{pmatrix}
1 & 0 \\ 0 & 11
\end{pmatrix}
\begin{pmatrix}
1 & 0 \\ -2 & 1
\end{pmatrix}
\begin{pmatrix}
1 & 1 \\ 0 & 1
\end{pmatrix},
\ldots, \\
&\begin{pmatrix}
-13 & 4 \\ 42 & -13
\end{pmatrix}
\begin{pmatrix}
1 & 0 \\ 0 & 11
\end{pmatrix}
\begin{pmatrix}
1 & -44 \\ 0 & 1
\end{pmatrix}
\begin{pmatrix}
66 & 5 \\ 13 & 1
\end{pmatrix}
\Big\} \\ 
= \Big\{ &\begin{pmatrix}
-13 & 4 \\ 42 & -13
\end{pmatrix}
\begin{pmatrix}
1 & 0 \\ 0 & 1
\end{pmatrix}
\begin{pmatrix}
1 & 0 \\ 0 & 11
\end{pmatrix}
\begin{pmatrix}
1 & 0 \\ 0 & 1
\end{pmatrix},
\begin{pmatrix}
-13 & 4 \\ 42 & -13
\end{pmatrix}
\begin{pmatrix}
1 & 0 \\ -22 & 1
\end{pmatrix}
\begin{pmatrix}
1 & 0 \\ 0 & 11
\end{pmatrix}
\begin{pmatrix}
1 & 1 \\ 0 & 1
\end{pmatrix}, \ldots, \\ &
\begin{pmatrix}
-13 & 4 \\ 42 & -13
\end{pmatrix}
\begin{pmatrix}
1 & -4 \\ 0 & 1
\end{pmatrix}
\begin{pmatrix}
1 & 0 \\ 0 & 11
\end{pmatrix}
\begin{pmatrix}
66 & 5 \\ 13 & 1
\end{pmatrix} = \begin{pmatrix}
-13 & 4 \\ 42 & -13
\end{pmatrix}
\begin{pmatrix}
1 & -4 \\ 0 & 1
\end{pmatrix}
\begin{pmatrix}
6 &  5 \\ 13 & 11
\end{pmatrix}
\begin{pmatrix}
11 & 0 \\ 0 & 1
\end{pmatrix}
\Big\}.
\end{align*}
We deduce that the Hecke matrix is $A = \begin{psmallmatrix}
0 & -1 & 2 \\ 4 & -4 & 3 \\ -1 & 1 & 4
\end{psmallmatrix}$. Again, the action of Hecke operators on points is given by evaluation of the $j$-invariant function at complex points or the modular polynomial $\Phi_{11}(X,Y)$.

\item \textbf{Algorithm \ref{alg:tiny_integral} and results:} In Step 1 of Algorithm \ref{alg:tiny_integral}, linear algebra over $\CC$ gives a power series expansion of the differential form $\omega_0$ at $j=j(Q)$:
\begin{align*}
    \omega_0 &= \frac{1}{3^4 \cdot 5^2 \cdot 13
} +  \frac{23}{3^{10} \cdot 5^5 \cdot 13}(j-j(Q)) +  \frac{4}{3^{13} \cdot 5^7 \cdot 13}(j-j(Q))^2 \\ &+ \frac{437174}{3^{22} \cdot 5^{10} \cdot 13^3
}(j-j(Q))^3 + \frac{138504533 }{3^{28} \cdot 5^{13} \cdot 13^4
}(j-j(Q))^4 + O((j-j(Q))^5) \ \ d(j-j(Q)).
\end{align*}

 Next, we compute the integrals as in Step 3. Repeating this for the other differentials and points, we record our results in Table~\ref{table:X_{ns}^+(13)_results}.

\begin{table}[h]
\label{table:X_{ns}_13}

    \centering
    \begin{tabular}{|l|l|}
    
    \hline
    \rule{0pt}{4ex}    
        $\sum_{i=0}^{11}\int^Q_{Q_i} \omega_0 $    & $10\cdot 11^{-1} + 9 + 9\cdot 11  + 6 \cdot 11^2 + 7\cdot 11^3 + 9\cdot 11^4 + O(11^5)$ 
            \rule{0pt}{4ex} \\
    
    \hline 
    \rule{0pt}{4ex}
    $\sum_{i=0}^{11}\int^Q_{Q_i} \omega_1 $ & $8\cdot 11^{-1} + 7 + 7\cdot 11 + 2 \cdot 11^2 + 6\cdot 11^3 + 6\cdot 11^4 + O(11^5)$
\\\hline

    \rule{0pt}{4ex}
    $\sum_{i=0}^{11}\int^Q_{Q_i} \omega_2 $ & $10\cdot 11^{-1} + 8 + 8\cdot 11 + 11^2  + 9\cdot 11^4 + O(11^5)$
\\\hline

    \rule{0pt}{4ex}
    $\sum_{i=0}^{11}\int^R_{R_i} \omega_0 $ & $7\cdot 11^{-1} + 2 + 3\cdot 11 + 9 \cdot 11^2  + 3\cdot 11^3 + 5\cdot 11^4 + O(11^5) $
\\\hline

    \rule{0pt}{4ex}
    $\sum_{i=0}^{11}\int^R_{R_i} \omega_1 $ & $6 + 6\cdot 11 + 11^3 + 5\cdot 11^4 + O(11^5)$ 
\\\hline

    \rule{0pt}{4ex}
    $\sum_{i=0}^{11}\int^R_{R_i} \omega_2 $ & $7\cdot 11^{-1} + 4 + 11 + 10 \cdot 11^2 +  10\cdot 11^3 + 5\cdot 11^4 + O(11^5)$ 
\\\hline
    \end{tabular}
    \caption{Coleman Integrations on $X_{\ns}^+(13)$ as in \S~\ref{sec:eg 13}.}
    \label{table:X_{ns}^+(13)_results}
\end{table}
\end{itemize}
\section{Remarks on computations}\label{section5}

\noindent \textbf{Choice of the upper half plane representative.}
When computing $\omega = \sum x_j (j-j(P))^i dj$, we compared Taylor expansions of both sides at $q = q(P)$ and used linear algebra over $\CC$ to recover the coefficients $x 
_i$. Therefore, the accuracy of the $x_i$'s depends on the convergence of the Taylor expansions. To achieve faster convergence, we want the imaginary part of $\tau(P)$ to be as large as possible.


In the case of $X_0^+(N)$, the Atkin-Lehner involutions can be used to perform the task. However, for a CM elliptic curve $E$ with discriminant $\Delta_E$, the situation is not ideal. Let $(c,d)$ be an integer solution to the norm equation $|c\tau_E + d|^2 = N$ and let $\hat{\gamma}$ be the lift of $(c,d)$ in $\slz$. Then the upper-half plane representative has imaginary part $\text{Im}(\hat{\tau}) = \text{Im}(\hat{\gamma} \cdot \tau_E) = \frac{\text{Im}(\tau_E)}{|c \tau_E + d|^2} =\frac{\sqrt{-\Delta_E}}{2}\cdot \frac{1}{N}$.

\noindent \textbf{Fast algorithm for differentiating $j$-function.} The computation of Taylor coefficients of the $j$-invariant function at $q_0 := e^{2 \pi i \tau_0}$ requires fast convergence. Since the coefficients of the $q$-expansion of  $j(q)$ are large, evaluating $j(q_0)$ and its derivatives $j^{(n)}(q_0)$ is an inefficient process, especially when $|q_0|$ is close to $1$. In \cite{cohen}, the $j$-function can be expressed as a rational function of modular functions:
\begin{align*}
a(q) &= 1+\sum_{n>0}(-1)^n(q^{n(3n-1)/2} + q^{n(3n+1)/2}) \\
\Delta(q) &= qa(q)^{24} \\
f(q) &= \frac{\Delta(2q)}{\Delta(q)} \\
j(q) &= \frac{(256f(q)+1)^3}{f(q)}
\end{align*}

Using these relations, the Taylor series of $j$ around $q = q_0$ can be expressed as rational functions of Taylor expansions of these modular functions.

\noindent \textbf{Choice of uniformizers.}
The modular polynomials $\Phi_p(X,j(Q))$ provides us an algebraic method to compute $j$-invariants of elliptic curves that are isogenous to $Q$. This motivates our use of rational combinations of the $j$-invariant function. The uniformizers $j:= j$ and $j := j + j_N$ were used in the above examples. One might consider alternative uniformizers such as $j := 1/j$ or
$j := j\cdot j_N$. There is experimental evidence that these options give smaller coefficients in Algorithm
\ref{alg:tiny_integral}.
\section{Future considerations}\label{section6}

We collect here some remarks that may pertain to future investigation of explicit Coleman integration on modular curves.

\subsection*{Integrating between CM points}

It was observed in Remark~\ref{remark:no exact term} that in contrast with the form of the ``fundamental linear system'' in other versions of explicit Coleman integration, the equation \eqref{eq:fundamental-eqn} is simpler because there is no contribution from the integrals of exact differentials.
Since we are free to vary the endpoints of a large integral within fixed residue discs, we may further simplify by requiring each of $Q$ and $R$ to be either a cusp or a CM point. The existence of a CM point in every noncuspidal residue disc is a consequence of Deuring's lifting theorem.

When $Q$ is a cusp, each of the $Q_i$ is also a cusp in the residue disc of $Q$, and hence is equal to $Q$ itself (because $p$ does not divide $N$).
As a corollary, we deduce that the integral of any holomorphic differentials between two cusps is zero, even if the cusps are in distinct residue discs; in other words, the difference between two cusps gives a torsion point in the Jacobian, which is the easy direction of the generalized Ogg's conjecture \cite[Conjecture~1.2]{yoo}.

When $Q$ is a CM point, each $Q_i$ is again a CM point of $Q$ in the residue disc of $Q$, but possibly with a different endomorphism ring; in particular, we may have $Q_i \neq Q$ (this is related to the \emph{isogeny volcanoes} described in \cite{bls10}). Nonetheless, it is conceivable that the Coleman integral $\int_{Q_i}^Q \omega$ admits an alternate interpretation that could be relevant for a model-free computation; in this vein, we recall that $j(Q)-j(Q_i)$ admits a simple closed-form expression \cite{gross-zagier}.

\subsection*{Iterated integrals}

As noted in the introduction, it is highly desirable to implement model-free iterated Coleman integration for modular curves (especially for double integrals, which occur in quadratic Chabauty). The paradigm of \cite{bbk10}, using the change of variables formula for a Frobenius lift, adapts directly to iterated integrals \cite{balakrishnan}.
However, our setup for Coleman integration on modular curves breaks down for iterated integrals at the reduction step: while the change of variables formula for single Coleman integrals extends from functions to correspondences, this is not true for higher Coleman integrals.

It should be possible to work around this using the fact that on a suitable wide open subspace of $X_0(N)$, the Hecke operator $T_p$ splits up canonically as the sum of the graph of a certain map $U_p$ and its transpose. 
In geometric terms, $U_p$ acts by quotienting the elliptic curve by a canonical subgroup of its $p$-torsion; this subgroup can be isolated on some strict neighborhood of the ordinary locus.

From the geometric description, it is clear that $T_p$ and $U_p$ commute, so eigenforms for the Hecke algebra are also eigenforms for $U_p$; this means that we can adapt \eqref{eq:fundamental-eqn} to $U_p$ without introducing any contributions from exact differentials.
It is also clear that $U_p$ carries cusps to cusps and CM points to CM points, so the previous remark continues to apply (with some care needed in the supersingular residue discs, on which $U_p$ is only defined near the boundary).

\subsection*{Local expansions at CM points}

Computing tiny integrals depends crucially on being able to expand differentials in power series in a chosen uniformizer. Since our current approach involves passing from complex approximations to $p$-adic approximations via algebraic reconstruction, some further work would be needed to make the precision estimates needed to put the calculations on a rigorous footing. This would in particular require estimating the degrees and heights of the algebraic numbers appearing as power series coefficients.
Alternatively, it may be feasible to use a different approach to obtain direct $p$-adic approximations of the power series coefficients.

\end{document}